\documentclass[11pt,letterpaper]{amsart}
\usepackage{amsmath}
\usepackage{amstext}
\usepackage{amssymb}
\usepackage{amsfonts}
\usepackage{enumerate}
\usepackage{mathrsfs}
\usepackage{braket}
\usepackage[dvipsnames]{xcolor}
\usepackage[all]{xy}
\usepackage{hyperref}

\numberwithin{equation}{section}

\newtheoremstyle{note}
{1em}
{1em}
{}
{}
{\bfseries}
{:}
{.5em}
{}

\newtheorem{theorem}{Theorem}[section]

\newtheorem{proposition}[theorem]{Proposition}
\newtheorem{corollary}[theorem]{Corollary}

\theoremstyle{note}
\newtheorem{remark}[theorem]{Remark}

\newtheorem{definition}[theorem]{Definition}
\newtheorem{example}[theorem]{Example}

\newcommand{\N}{{\mathbb{N}}}
\newcommand{\R}{{\mathbb{R}}}
\newcommand{\C}{{\mathbb{C}}}

\DeclareMathOperator{\tr}{tr}
\DeclareMathOperator{\spa}{span}
\DeclareMathOperator{\rank}{rank}

\title{Connectivity for Quantum Graphs}

\author[J.A. Ch\'avez-Dom\'inguez]{Javier Alejandro Ch\'avez-Dom\'inguez}
\address{Department of Mathematics, University of Oklahoma, Norman, OK 73019-3103,
USA} \email{jachavezd@ou.edu}

\author{Andrew T. Swift}
\email{ats0@ou.edu}

\thanks{The first author was partially supported by NSF grant DMS-1900985.}


\begin{document}

\maketitle

\begin{abstract}
In quantum information theory there is a construction for quantum channels, appropriately called a quantum graph, that generalizes the confusability graph construction for classical channels in classical information theory.  In this paper, we provide a  definition of connectedness for quantum graphs that generalizes the classical definition. This is used to prove a quantum version of a particular case of the classical tree-packing theorem from graph theory. Generalizations for the related notions of $k$-connectedness and of orthogonal representation are also proposed for quantum graphs, and it is shown that orthogonal representations have the same implications for connectedness as they do in the classical case. 
\end{abstract}

\section{Introduction}
In classical zero-error information theory, one is interested in the accurate transmission and recovery of messages through a noisy channel.  Typically these messages are transmitted one letter of the alphabet at a time and properties of the transmission needed to ensure an accurate reading of the message (such as repetition of a sent letter) are determined from the noise of the channel.
To model this sort of scenario, we consider finite sets $V$ and $W$ that represent the input and output alphabets, respectively. A \emph{classical channel} consists of choosing for each input $v \in V$ 
a probability distribution over $W$, specifying how $v$ might be read after transmission through the channel; this represents the noise of the channel. 
The accuracy of a sent message boils down to how likely two different input letters might be transmitted and then received as the same output. 
Thus, a natural graph-theoretical construction that we can associate to a channel as above is the graph having elements in $V$ as vertices and where $u, v \in V$ are connected by an edge if there is positive probability that $u$ and $v$ could be transmitted and received as the same output.  This graph is called the \emph{confusability graph} of the channel, and it is not hard to see that every (finite) graph (with all possible loops) can be realized as the confusability graph of some channel.  In this way, there is a rich interplay between graph theory and information theory.

The purpose of this paper is to study the connectivity of the analogue of confusability graph that arises naturally from \emph{quantum} information theory (see \cite{Duan-Severini-Winter}).
To better motivate the definition of a quantum channel, observe first that a classical channel as described in the previous paragraph is canonically associated with a linear map $\R^V \to \R^W$:  For each $v \in V$, the vector having a $1$ in the $v$-th position and zeroes everywhere else gets mapped to the probability density associated to $v$, and this map is then extended linearly. Observe that this linear map is positive (that is, it sends nonnegative vectors to nonnegative vectors) and moreover it maps probability densities to probability densities.
In quantum information theory, the role of a probability density is played by a 
quantum state, 
that is, a positive semidefinite matrix with trace 1.
A \emph{quantum channel} is then represented by a linear map $\Phi\colon M_n \to M_m$ between spaces of matrices with complex entries, which is trace-preserving and \emph{completely positive}; the latter term means that not only is the map $\Phi$ positive (i.e. it maps positive semidefinite matrices to positive semidefinite matrices), but also the same is true whenever we take the tensor product of $\Phi$ with the identity mapping on $M_k$ for each $k\in\N$.
By Choi's theorem (\cite{Choi}), since $\Phi\colon M_n \to M_m$ is completely positive there exist matrices $K_1, K_2,\dots K_N \in M_{m,n}$ such that $\Phi(\rho)=\sum_{i=1}^N K_i\rho K_i^\dagger$ for all matrices $\rho\in M_n$. 
In the quantum setting, two transmitted states $\rho$ and $\psi$ are distinguishable from each other if their images are orthogonal, and this may be seen to be equivalent to the condition that $\rho A \psi=0$ for all $A\in  \mathrm{span}\{K_i^\dagger K_j\}_{1\leq i,j\leq N}$ \cite{Duan-Severini-Winter}.
For this reason, and by analogy to the classical setting, $\mathrm{span}\{K_i^\dagger K_j\}_{1\leq i,j\leq N}$ is called the \emph{quantum confusability graph} associated to $\Phi$. It is easy to see that a quantum confusability graph is an \emph{operator system}, that is, a linear space of matrices with complex entries which is closed under taking adjoints and contains the identity matrix (since $\Phi$ is trace-preserving, $\sum_{i=1}^N K_i^\dagger K_i=\mathrm{Id}$), and in fact every operator system can be realized as the quantum confusability graph of some quantum channel \cite{duan2009super,Cubitt-Chen-Harrow}.
With the motivation given above, and despite several other strong contenders for the title, we follow \cite{weaver2015quantum} in using the terminology \emph{quantum graph} rather than operator system to emphasize the graph-theoretical flavor of our investigations.  Indeed, even without the connection to quantum information theory, there is already good justification for doing this (\cite{KuperbergWeaver}).  

It is our hope to expand the toolbox available to quantum information theorists by discovering the limits of what methods can be transferred from the well-understood classical graph theory setting into the quantum one; results of this nature have already appeared in works such as \cite{Duan-Severini-Winter,Stahlke,Ortiz-Paulsen,Weaver-quantum-ramsey,Levene-Paulsen-Todorov,Kim-Mehta,weaver2018quantum}.
There are many important classical graph-theoretical concepts that deserve investigation, and if any of these have a good quantum analogue, it can be reasonably expected that they possess a utility similar to their classical counterparts.  One of the most fundamental of these concepts is the notion of connectedness.  We provide a natural definition of quantum connectedness for quantum graphs that generalizes the classical one, and explore what extensions/analogues of classical connectivity theorems hold true in the quantum setting.

\section{Notation}

We denote the space of all $k$ by $n$ matrices with complex entries by $M_{k,n}$,  or by $M_n$ if $k=n$. 
We let $X^\dagger$ denote the Hermitian adjoint of a matrix $X\in M_{k,n}$ and let $\|X\|$ denote the operator norm of $X$ , so that $\|X\|^2$ is the largest eigenvalue of $ X^\dagger X$.
We equip $M_n$ with the inner product given by $\langle X,Y\rangle = \tr(X^\dagger Y)$, 
where $\tr(Z)$ is the 
trace of a matrix $Z\in M_n$. 
We write $I_n$ (or simply~$I$) for the identity matrix in $M_n$. 
A projection is $P \in M_n$ such that $P = P^2 = P^\dagger $, and a nontrivial projection is  a projection which is neither zero nor $I_n$.
We use Dirac's bra-ket notation: $\ket{u} \in \C^n=M_{n,1}$ is a vector, $\bra{u}=\ket{u}^\dagger \in M_{1,n}$ is its adjoint (a linear form), $\braket{u|v}$ is the standard Hilbert space inner product (linear in the second argument) of $u$ and $v$, and $\ket{v}\bra{u}\in M_n$ is the corresponding rank-one operator defined by $\ket{v}\bra{u}(\ket{w})=\braket{u|w}\ket{v}$.  The list $(\ket{e_k})_{k=1}^n$ will always denote the standard basis of $\mathbb{C}^n$.
The cardinality of a set $S$ is denoted by $|S|$.
For $n\in\N$, $[n]$ denotes the set $\{1,2, \dotsc, n\}$.

By a \emph{quantum graph} on $M_n$ we mean an operator system:  A linear subspace of $M_n$ which is closed under taking adjoints and contains the identity matrix.
To any classical graph $G$ with vertex set $[n]$ we associate the quantum graph
\[
\mathcal{S}_G = \spa\big\{  \ket{e_i}\bra{e_j} \mid i=j \text{ or $i$ is adjacent to $j$}\big\} \subseteq M_n.
\]

Given two quantum graphs $\mathcal{U}, \mathcal{V} \subseteq M_n$, by their product we mean 
\[
\mathcal{U}\mathcal{V} =  \spa\big\{ UV \mid U \in \mathcal{U}, V \in \mathcal{V} \big\},
\]
and  we define $\mathcal{U}^m$ for $m\in \mathbb{N}\cup \{0\}$ recursively by
$$
\mathcal{U}^0=\mathbb{C}I_n, \qquad  \mathcal{U}^{k+1}=\mathcal{U}^k\mathcal{U}.
$$
Note that the product of quantum graphs is a quantum graph.

To emphasize the distinction between quantum and non-quantum graphs, we use the adjective \emph{classical} when we are talking about a combinatorial graph. 
We use the notation $i \sim j$ to indicate that two vertices $i$ and $j$ are adjacent in a classical graph.

\section{Connectedness}
In this section, we define what it means for a quantum graph to be ``connected'' and show some equivalences that highlight the similarity to classical connectedness, including a quantum analogue of the base case of the classical tree-packing theorem.  In particular, we show that a classical graph is connected if and only if its associated quantum graph is connected.  Philosophically, the ``vertices'' in a quantum graph correspond to rank one projections, and collections of vertices correspond to possibly higher rank projections.  Because of this, the main obstacle for directly adapting a classical graph concept to quantum graph theory is that we should require such concepts to be coordinate-free.  Indeed, if an orthonormal basis is fixed, there are natural classical graphs that can be associated to any quantum graph such that collections of vertices correspond to projections whose images align with the basis.  We will show that for a connected quantum graph, any choice of orthonormal basis will give rise to a connected classical graph.

The following definition of connectedness is based on the intuition that in a connected graph, there is a path between any two vertices.

\begin{definition}
\label{defn-connected}
A quantum graph $\mathcal{S} \subseteq M_n$ is \emph{connected} if there exists $m\in\N$ such that $\mathcal{S}^m = M_n$.
A quantum graph which is not connected will be called \emph{disconnected}.
\end{definition}

\begin{example}[The quantum hamming cube is connected]
The quantum Hamming cube \cite[Defn. 3.7]{KuperbergWeaver} is the quantum graph 
\[
\mathcal{C}_n = \spa \bigg\{ \bigotimes_{i=1}^nA_i \mid A_i \in M_2, \text{all but one of the $A_i$ are equal to $I_2$} \bigg\} \subseteq M_{2^n}.
\]
Notice that $\mathcal{C}_n^n$ contains
\[
 \spa \bigg\{ \bigotimes_{i=1}^n A_i \mid A_i \in M_2 \bigg\} = M_{2^n},
\]
so $\mathcal{C}_n$ is connected.
\end{example}

Another intuitive condition that we could have used to motivate a definition of connectedness is that in a disconnected classical graph, there is always a partition of the set of vertices into two nonempty pieces such that the two pieces have no edge between them.  The next theorem shows that the quantum analogue of this condition is equivalent to our definition of connectedness.

\begin{theorem}\label{thm-char-connectivity}
Let $\mathcal{S} \subseteq M_n$ be a quantum graph.
Then $\mathcal{S}$ is disconnected if and only if there exists a nontrivial projection $P \in M_n$ such that ${P \mathcal{S} (I_n-P)} = \{0\}$.
\end{theorem}

\begin{proof}
Suppose there exists a projection $P \in M_n \setminus\{0,I_n\}$ such that $P \mathcal{S} (I_n-P) = \{0\}$.
Since $\mathcal{S}$ is closed under taking adjoints,
 for every $A \in \mathcal{S}$ we have $PA(I_n-P) =0$ and also $ (I_n-P)AP = (PA^\dagger (I_n-P))^\dagger = 0$.
It follows that 
\[A =(P+(I_n-P))A(P+(I_n-P))= PAP + (I_n-P)A(I_n-P).\]
Thus, for any $A,B\in \mathcal{S}$,
\[
AB = PAPBP + (I_n-P)A(I_n-P)B(I_n-P).
\]
It follows that $PAB(I_n-P)=0$, and so $P\mathcal{S}^2(I_n-P)=\{0\}$, and similarly $P\mathcal{S}^m(I_n-P)=\{0\}$ for any $m\in\N$.
Therefore $\mathcal{S}^m\neq M_n$ for any $m\in \mathbb{N}$.
That is, $\mathcal{S}$ is disconnected.

Now suppose that $\mathcal{S}$ is disconnected.
Since $(\mathcal{S}^m)_{m=1}^\infty$ is an increasing sequence of proper subspaces of the finite-dimensional space $M_n$, it must stabilize at a proper subspace.
Note then that $\mathcal{A} = \bigcup_{m=1}^\infty \mathcal{S}^m$ is a proper $C^*$-subalgebra of $M_n$.
From well-known classical results about the structure of finite-dimensional $C^*$-algebras \cite[Thm. III.1.1 and Cor. III.1.2]{Davidson},
there exist nontrivial disjoint projections $P_1, \dotsc, P_k$ in $M_n$ adding up to $I_n$ such that 
$\mathcal{A} = P_1\mathcal{A}P_1 \oplus \cdots \oplus  P_k\mathcal{A}P_k$.
Any of the projections $P_j$ will then satisfy $P_j \mathcal{A} (I_n-P_j) = \{0\}$,
so in particular $P_1 \mathcal{S} (I_n-P_1) = \{0\}$.
\end{proof}

As a consequence of Theorem \ref{thm-char-connectivity}, quantum connectedness generalizes classical connectedness.

\begin{corollary}
Let $G$ be a classical graph with vertex set $[n]$ and associated quantum graph $\mathcal{S}_G$.
Then $G$ is connected if and only if $\mathcal{S}_G$ is connected.
\end{corollary}

\begin{proof}
Suppose $G$ is connected.  Then for each $i,j\in [n]$, there is a path $(p_k)_{k=1}^m$ in $G$ such that $p_1=i$, $p_m=j$, and $m\leq n$.  But this means $\ket{e_{p_k}}\bra{e_{p_{k+1}}}\in S_G$ for all $1\leq k\leq m-1$ and so $\ket{e_i}\bra{e_j} = \prod_{k=1}^{m-1}\ket{e_{p_k}}\bra{e_{p_{k+1}}}\in S_G^m\subseteq S_G^n$.  As $\{\ket{e_i}\bra{e_j}\}_{1\leq i,j\leq n}$ forms a basis for $M_n$, this implies $S_G^n=M_n$, and so $S_G$ is connected.

On the other hand, suppose $G$ is disconnected.  Then $[n]$ can be partitioned into two nonempty sets $K$ and $L$ that are not connected to each other by any edge in $G$. This implies that $\ket{e_i}\bra{e_j}$ and $\ket{e_j}\bra{e_i}$  are orthogonal to $S_G$ whenever $i\in K$ and $j\in L$.  Thus, if $P = \sum_{j \in K} \ket{e_j}\bra{e_j}$ is the orthogonal projection onto $\spa\{ e_j\}_{j \in K}$, then $P \mathcal{S}_G (I_n-P) = \{0\}$.  And so, by Theorem \ref{thm-char-connectivity}, $S_G$ is disconnected.
\end{proof}

Observe that a different way of stating Theorem \ref{thm-char-connectivity} is the following:  A quantum graph $\mathcal{S} \subseteq M_n$ is connected if and only if 
whenever $P_1,P_2$ are nontrivial disjoint projections adding up to $I_n$, we have $\dim P_1\mathcal{S}P_2 + \dim P_2\mathcal{S}P_1 \ge 2$.
This suggests a quantum version of the following particular case of the tree-packing theorem of Tutte \cite{Tutte} and Nash-Williams \cite{Nash-Williams}:  A classical graph contains a spanning tree (i.e., it is connected) if and only if every partition $\mathscr{P}$ of its vertex set has at least $|\mathscr{P}|-1$ cross-edges (that is, edges joining two vertices that belong to different pieces of the partition).

\begin{theorem}\label{thm-TNW-case-k-equals-1}
A quantum graph $\mathcal{S} \subseteq M_n$ is connected if and only if $\sum_{i\not=j} \dim \big[P_j \mathcal{S} P_i\big] \ge 2(m-1)$ whenever $P_1,\dotsc,P_m$ are nontrivial disjoint projections adding up to the identity.
\end{theorem}

\begin{proof}
Suppose first that $P \in M_n$ is a nontrivial projection such that
\[
\dim \big[P\mathcal{S}(I_n-P)\big] + \dim \big[(I_n-P)\mathcal{S}P] \ge 2.
\]
Then $\dim \big[P\mathcal{S}(I_n-P)\big] = \dim \big[(I_n-P)\mathcal{S}P\big]=1$, because $\mathcal{S}$ is closed under taking adjoints, and so $P\mathcal{S}(I_n-P) \not= \{0\}$.  By Theorem \ref{thm-char-connectivity}, $\mathcal{S}$ is connected.

Now assume that $\mathcal{S}$ is connected, and let $P_1,\dotsc,P_m$ be nontrivial disjoint projections adding up to the identity. 
Define a classical graph $G$ on $[m]$ via $i \sim j$ if and only if $P_i \mathcal{S} P_j \not= \{0\}$.
We claim that $G$ is a connected classical graph. Otherwise, we can partition $[m]$ into disjoint nonempty subsets $A$ and $B$ such that $i \in A$ and $j\in B$ implies $P_i \mathcal{S} P_j = \{0\}$.
But this would imply
\[
\Big( \sum_{i\in A} P_i  \Big) \mathcal{S} \Big( \sum_{j\in B} P_j \Big) = \{0\},
\]
contradicting the fact that $\mathcal{S}$ is connected, by Theorem \ref{thm-char-connectivity}.
Since $G$ is connected it must have at least $m-1$ edges, which implies $\sum_{i<j} \dim \big[P_i \mathcal{S} P_j\big] \ge m-1$.
\end{proof}

For any classical graph $G$, there is a canonical quantum graph $\mathcal{S}_G$ associated to $G$.  To go in the other direction and associate a classical graph to a given quantum graph, an orthonormal basis (o.n.b) for $\mathbb{C}^n$ must first be chosen.  If $v=(\ket{v_k})_{k=1}^n$ is an (ordered) o.n.b for $\mathbb{C}^n$, 
{ then one of the most natural classical graphs we can associate to a quantum graph $\mathcal{S}$ with respect to $v$ is the graph $C_v(\mathcal{S})$ with vertex set $[n]$, where $i,j\in [n]$ are adjacent exactly when $\bra{v_i}A\ket{v_j}\neq 0$ for some $A\in \mathcal{S}$.
We call $C_v(\mathcal{S})$ the \emph{confusability graph of $\mathcal{S}$ with respect to $v$} 
(note that our terminology does not agree with that of \cite{Kim-Mehta}). }
It is not hard to see that if $v$ is the standard basis, then $C_v(\mathcal{S}_G)=G$ for any classical graph $G$, and it is this property that informs our choice of graph construction.  We have already seen that quantum connectedness is a generalization of classical connectedness.  Even so, the following proposition allows us to rephrase quantum connectedness in terms of classical connectedness.

\begin{proposition}\label{prop-connectivity-distinguishability}
Let $\mathcal{S} \subseteq M_n$ be a quantum graph.
Then $\mathcal{S}$ is connected if and only if $C_v(\mathcal{S})$ is connected for every o.n.b. $v = (\ket{v_k})_{k=1}^n$ of $\C^n$.
\end{proposition}

\begin{proof}
Suppose $\mathcal{S}$ is disconnected, so that by Theorem \ref{thm-char-connectivity} there exists a nontrivial projection $P \in M_n$ such that $P \mathcal{S} (I_n-P) = \{0\}$.
Let $v = (\ket{v_k})_{k=1}^n$ be an o.n.b. of $\C^n$ such that for some $1\leq m<n$, $(\ket{v_k})_{k=1}^m$ is an o.n.b. for the range of $P$, and hence $(\ket{v_k})_{k=m+1}^n$ is an o.n.b for the range of $I_n-P$. For each $1 \le i \le m$, $m+1 \le j \le n$, and $A \in \mathcal{S}$ we then have
$0 = \ket{v_i}\bra{v_i} A \ket{v_j}\bra{v_j}$, which implies 
$\braket{v_i | A | v_j} = 0$
and therefore $i \not \sim j$ in $C_v(\mathcal{S})$, showing that $C_v(\mathcal{S})$ is disconnected.

Suppose now that there exists $v = (\ket{v_k})_{k=1}^n$ an o.n.b. of $\C^n$ such that $C_v(\mathcal{S})$ is disconnected.
Let $K,L$ partition $[n]$ into disjoint nonempty sets such that for all $i\in K$ and $j \in L$ we have $i \not \sim j$ in  $C_v(\mathcal{S})$,
that is, $\braket{v_i | A | v_j} = 0$ for all $A \in \mathcal{S}$.
Set $P = \sum_{i\in K} \ket{v_i}\bra{v_i}$, so that $I_n - P = \sum_{j\in L} \ket{v_j}\bra{v_j}$ and thus, for every $A \in \mathcal{S}$,
\[
PA(I_n-P) = \sum_{i\in K, j \in L} \ket{v_i}\bra{v_i} A \ket{v_i}\bra{v_i} = 0,
\]
which implies that $\mathcal{S}$ is disconnected,  by Theorem \ref{thm-char-connectivity}.
\end{proof}

\section{$k$-connectedness}
In the previous section, we defined a notion of connectedness for quantum graphs that generalizes the notion of connectedness for classical graphs.  In this section, we provide a measure for the amount of connectedness a quantum graph has by way of a quantum analogue of connectivity.  In the classical case, the connectivity of a graph $G$ is the number of vertices that would have to be removed from $G$ to create a graph that either is disconnected or contains a single vertex.  This idea can be mimicked for quantum graphs once one considers how to properly define a notion of creating a ``subgraph'' by ``removal of vertices''.

In other words, what is needed is a notion of \emph{restriction}:  Given a quantum graph and a ``subset of vertices'', we would like to define the ``subgraph'' obtained when we restrict our attention to the given subset.
This has already been considered by Weaver in the more general setting of quantum relations \cite[Sec. 3]{weaver2015quantum}, and we adopt the same definition.

Concretely, given a quantum graph $\mathcal{S}\subseteq M_n$ and a projection $P\in M_n$, we consider $P\mathcal{S}P$ to be a subgraph of $\mathcal{S}$ restricted to $M_{\rank(P)}\cong PM_nP$ 
(it is easy to check that $P\mathcal{S}P$ is indeed a quantum graph).
We are led to the following definitions.

{
\begin{definition}
\label{defn-separator}
Let $\mathcal{S} \subseteq M_n$ be a quantum graph.  A projection $P\in M_n$ is called a \emph{separator} of $\mathcal{S}$ if $(I_n-P)\mathcal{S}(I_n-P)$ is either disconnected (viewed as a subspace of $M_{n-\rank(P)}$) or 1-dimensional.
\end{definition}
}

\begin{remark}\label{defn-separators}
By Definition \ref{defn-connected}, a projection $P$ such that $\rank(P)<n-1$ is a separator for a quantum graph $\mathcal{S}\subseteq M_n$ if and only if there is no $m\in \mathbb{N}$ such that
\[
\big((I_n-P)\mathcal{S}(I_n-P)\big)^m = (I_n-P)M_n(I_n-P).
\]

Theorem \ref{thm-char-connectivity} provides another characterization:  A projection $P$ such that $\rank(P)<n-1$ is a separator if and only if there exist nontrivial projections $Q_1$ and $Q_2$, disjoint from each other and from $P$, such that $Q_1+Q_2=I_n-P$ and $Q_1\mathcal{S}Q_2=\{0\}$.

We will use whichever property of separator is most useful in what is to follow.
\end{remark}

\begin{definition}
Let $k\in \N$.  A quantum graph $\mathcal{S} \subseteq M_n$ is called \emph{$k$-connected} if every separator for $\mathcal{S}$ has rank at least $k$.
\end{definition}

In particular, a quantum graph $\mathcal{S}$ on $M_n$ is connected if and only if either $\mathcal{S}$ is 1-connected or $\mathcal{S}=M_n$.

Let us now prove that this quantum notion of $k$-connectedness generalizes the classical one.

\begin{proposition}
Let $G$ be a classical graph with vertex set $[n]$ and associated quantum graph $\mathcal{S}_G$, and let $k\in\N$.
Then $G$ is $k$-connected if and only if $\mathcal{S}_G$ is $k$-connected.
\end{proposition}

\begin{proof}

Suppose $\mathcal{S}_G$ is $k$-connected.  If $G$ is a complete graph, then $G$ is $k$-connected.  So suppose $G$ is not a complete graph.  Let $\{p_i\}_{i=1}^m\subseteq [n]$ be a vertex cut of $G$ that induces a disconnected subgraph of $G$.  Then $P=\sum_{i=1}^m \ket{p_i}\bra{p_i}$ is a separator of rank $m$ for $\mathcal{S}_G$.  Thus $m\geq k$, which implies that $G$ is $k$-connected.

Suppose now that $G$ is $k$-connected.  If every separator of $\mathcal{S}_G$ has rank at least $n-1$, then $\mathcal{S}_G$ is $k$-connected by definition. So suppose there is a separator $P \in M_n$  such that $\rank(P)<n-1$.   Then there exist nontrivial disjoint projections $Q_1, Q_2 \in M_n$, also disjoint from $P$, such that $I_n = P + Q_1 + Q_2$ and $Q_1\mathcal{S}_GQ_2 = \{0\}$.
Let $(\ket{v_i})_{i=1}^n$ be an orthonormal basis for $\C^n$ which consists of the union of some orthonormal bases for the ranges of $P$, $Q_1$, and $Q_2$. By permuting the indices if necessary, it follows from \cite[Lemma 13]{Kim-Mehta} (which in turn is \cite[Lemma 7.28]{paulsen2016entanglement}) that we can assume $\braket{v_i|e_i} \not=0$ for each $1\le i \le n$.
Let $K,L_1, L_2$ be disjoint subsets of $[n]$ such that $L_1,L_2$ are nonempty and $K\cup L_1\cup L_2=[n]$, and such that 
\[
P = \sum_{i\in K} \ket{v_i}\bra{v_i}, \quad Q_1 = \sum_{i\in L_1} \ket{v_i}\bra{v_i}, \quad Q_2 = \sum_{i\in L_2} \ket{v_i}\bra{v_i}.
\]
Notice that if $k \sim l$ in $G$ (i.e., $\ket{e_k}\bra{e_l}\in \mathcal{S}_G$), then
\[
0 = Q_1 \ket{e_k}\bra{e_l} Q_2 = \sum_{i \in L_1, j \in L_2} \ket{v_i}\braket{v_i|e_k}\braket{e_l|v_j} \bra{v_j},
\]
and thus $\braket{v_i|e_k}\braket{e_l|v_j} = 0$ for each $i \in L_1$ and $j\in L_2$.
But for $i \in L_1$ and $j\in L_2$ we have $\braket{v_i|e_i}\braket{e_j|v_j} \not= 0$, which means $i \not\sim j$ in $G$.
Since $G$ is $k$-connected, this implies $k \le |K| = \rank(P)$, showing that $\mathcal{S}_G$ is $k$-connected.
\end{proof}

Just as in the case of connectedness (see Proposition \ref{prop-connectivity-distinguishability}), $k$-connectedness of a quantum graph is equivalent to the $k$-connectedness of all its confusability graphs.

\begin{proposition}\label{prop-k-connectivity-distinguishability}
Let $\mathcal{S} \subseteq M_n$ be a quantum graph, and $k\in\N$.
Then $\mathcal{S}$ is $k$-connected if and only if $C_v(\mathcal{S})$ is $k$-connected for every o.n.b. $v = (\ket{v_i})_{i=1}^n$ of $\C^n$.
\end{proposition}

\begin{proof}
Suppose $\mathcal{S}$ is $k$-connected, and let $(\ket{v_i})_{i=1}^n$ be an orthonormal basis for $\C^n$.  If $C_v(\mathcal{S})$ is a complete graph, then $C_v(\mathcal{S})$ is $k$-connected.  So suppose $C_v(\mathcal{S})$ is not a complete graph.
Let $K,L_1,L_2$ be disjoint subsets of $[n]$ such that $L_1, L_2$ are nonempty and $K\cup L_1\cup L_2=[n]$, and such that there are no edges in $C_v(\mathcal{S})$ between $L_1$ and $L_2$.
Notice that this means for each $i\in L_1$, $j \in L_2$ and $A \in \mathcal{S}$ we have $\braket{v_i | A | v_j} = 0$.
Define
\[
P = \sum_{i\in K} \ket{v_i}\bra{v_i}, \quad Q_1 = \sum_{i\in L_1} \ket{v_i}\bra{v_i}, \quad Q_2 = \sum_{i\in L_2} \ket{v_i}\bra{v_i}.
\]
Notice that for each $A \in \mathcal{S}$
\[
Q_1 A Q_2 = \sum_{i \in L_1, j \in L_2} \ket{v_i}\braket{v_i | A | v_j} \bra{v_j} =0.
\]
Therefore $P$ is a separator for $\mathcal{S}$, so $k \le \rank(P) = |K|$, showing that $C_v(\mathcal{S})$ is $k$-connected.

Assume now that $C_v(\mathcal{S})$ is $k$-connected for every o.n.b. $v$ of $\mathbb{C}^n$.  If every separator of $\mathcal{S}_G$ has rank at least $n-1$, then $\mathcal{S}_G$ is $k$-connected by definition.  So suppose there is a separator $P \in M_n$ such that $\rank(P)<n-1$.  Then there exist nontrivial disjoint projections $Q_1, Q_2 \in M_n$, also disjoint from $P$, such that $I_n = P + Q_1 + Q_2$ and $Q_1\mathcal{S}Q_2 = \{0\}$.
Let $v=(\ket{v_i})_{i=1}^n$ be an orthonormal basis for $\C^n$ which consists of the union of some orthonormal bases for the ranges of $P$, $Q_1$ and $Q_2$, and let $K,L_1,L_2$ be disjoint subsets of $[n]$ such that $L_1,L_2$ are nonempty and $K\cup L_1\cup L_2=[n]$, and such that
\[
P = \sum_{i\in K} \ket{v_i}\bra{v_i}, \quad Q_1 = \sum_{i\in L_1} \ket{v_i}\bra{v_i}, \quad Q_2 = \sum_{i\in L_2} \ket{v_i}\bra{v_i}.
\]
Notice that for each $A \in \mathcal{S}$ we have
\[
0 = Q_1 A Q_2 = \sum_{i \in L_1, j \in L_2} \ket{v_i}\braket{v_i | A | v_j} \bra{v_j},
\]
which implies that
for each $i\in L_1$, $j \in L_2$ and $A \in \mathcal{S}$ we have $\braket{v_i | A | v_j} = 0$.
But this means that there are no edges in $C_v(\mathcal{S})$ between $L_1$ and $L_2$,
so by the $k$-connectivity of $C_v(\mathcal{S})$ we conclude $k \le |K| = \rank(P)$ and therefore $\mathcal{S}$ is $k$-connected.
\end{proof}

In the classical setting,  a graph on $n$ vertices is $(n-1)$-connected if and only if it is complete.
In the quantum setting this is no longer true, but we can still characterize the maximally connected quantum graphs.
\begin{proposition}\label{prop-maximal-connectivity}
Let $\mathcal{S} \subseteq M_n$ be a quantum graph. Then $\mathcal{S}$ is $(n-1)$-connected if and only if $A\mathcal{S}B \not=\{0\}$ for every $A,B \in M_n\setminus\{0\}$.
\end{proposition}
\begin{proof}
Suppose that $A\mathcal{S}B \not=\{0\}$ for every $A,B \in M_n\setminus\{0\}$.
It follows from Remark \ref{defn-separators} that $\mathcal{S}$ does not admit a separator of rank strictly smaller than $n-1$, and therefore $\mathcal{S}$ is $(n-1)$-connected.
Suppose, on the contrary, that there exist $A,B \in M_n\setminus\{0\}$ such that $A\mathcal{S}B =\{0\}$.  Let $v$ be a vector in the range of $B$ and $u$ a vector in the range of $A^\dagger$.  
Then $I_n-\ket{u}\bra{u}-\ket{v}\bra{v}$ is a separator for $\mathcal{S}$ with rank $n-2$, and therefore $\mathcal{S}$ is not $(n-1)$-connected.
\end{proof}
It is not difficult to produce examples of quantum graphs contained in $M_n$ satisfying the condition in the previous Proposition without being all of $M_n$.
An example for $n=2$ is provided at the beginning of \cite[Sec. 4]{weaver2015quantum}, and more generally one can consider
\[
\spa\big\{ I_n, \ket{e_i}\bra{e_j} \mid 1\le i,j \le n, i \not= j \big\} \subsetneq M_n.
\]


\section{Orthogonal representations}
With a definition of  $k$-connectedness that generalizes the classical notion, the next order of business is to find sufficient conditions for a quantum graph to be $k$-connected.  One motivation for such a condition comes from the classical realm in the form of orthogonal representations of graphs (see \cite{Lovasz-Saks-Schrijver}).
Recall that for a classical graph $G=(V,E)$, an \emph{orthogonal representation} is an assignment $f\colon V \to \R^d$ or $f\colon V \to \C^d$ such that for every $i,j \in V$ with $i\not=j$,
\[
i \not\sim j \quad \Rightarrow \quad f(i) \perp f(j).
\]

An orthogonal representation $f$ of $G=(V,E)$  is said to be in \emph{general position} if for any $U \subseteq V$ such that $|U|=d$, the vectors in $\{f(i)\}_{i \in U}$ are linearly independent.
A weaker condition is to require only that the vectors representing the vertices nonadjacent to any fixed vertex are linearly independent.  For brevity, we will say that such a representation is in \emph{locally general position}. 

The relationship between these notions and connectivity is given by Theorem 1.1' in \cite{Lovasz-Saks-Schrijver}:

\begin{theorem}\label{thm-LSS}
If $G$ is a classical graph with $n$ vertices, then the following are equivalent:
\begin{enumerate}[(a)]
  \item $G$ is $(n-d)$-connected.
    \item $G$ has an orthogonal representation in $\R^d$ in general position.
    \item $G$ has an orthogonal representation in $\R^d$ in locally general position.
\end{enumerate}
\end{theorem}

Our desire is to find a condition such as (b) or (c) in the above theorem that will imply some amount of connectivity for a quantum graph.  We start by considering what it should mean for a quantum graph to be ``orthogonally represented'', motivated by the concept of order zero maps.  Recall the following definition.

\begin{definition}Let $A,B$ be $C^*$-algebras.
\begin{enumerate}[(a)]
    \item Two elements $a,b\in A$ are called \emph{orthogonal}, denoted $a \perp b$, if $0=ab=ba=a^*b = ab^*$.
    \item A completely positive map $\phi\colon A \to B$ is said to be \emph{order zero} if $\phi(a)\perp \phi(b)$ whenever $a\perp b$.
\end{enumerate}
\end{definition}

Order zero maps are known to have a nice structure, see
\cite[Thm. 1.2]{WZ-nuclear-dimension} and
\cite[Thm. 2.3]{WZ-order-zero}.  In terms of quantum graphs, $\phi\colon M_n\to M_d$ is order zero if and only if $\phi(A)\perp \phi(B)$ whenever $A$ and $B$ are ``nonadjacent'' in the quantum graph $\mathbb{C}\cdot I_n$.  We are led by analogy to the following definition.

\begin{definition}
\label{def:orth}
Let $\mathcal{S} \subseteq M_n$ be a quantum graph. A completely positive map $\phi\colon M_n \to M_d$ is said to be an \emph{orthogonal representation} of $\mathcal{S}$ if $\phi(A) \perp \phi(B)$ for any $A, B \in M_n$  such that
\[
A \mathcal{S} B = B \mathcal{S} A = A^* \mathcal{S} B = A \mathcal{S} B^*=  \{0\}.
\]
\end{definition}

Note that if $\mathcal{S} \subseteq M_n$ is a quantum graph, then the identity map $I_n\colon M_n \to M_n$ is trivially an orthogonal representation of $\mathcal{S}$. 
Definition \ref{def:orth} is justified by the following two propositions. 
\begin{proposition}\label{prop-ortho-rep-is-order-zero}
Let $G=(V,E)$ be a classical graph with $n$ vertices and let $f\colon V \to \C^d$ be an orthogonal representation of $G$. Then the completely positive map $\phi\colon M_n \to M_d$ defined by
\[
\phi(X) = \sum_i  \ket{f(i)}\bra{e_i} X \ket{e_i}\bra{f(i)} \quad \text{ for all } X \in M_n
\] is an orthogonal representation of $\mathcal{S}_G$.
\end{proposition}

\begin{proof}
Pick any $A,B\in M_n$ such that
\[
A \mathcal{S}_G B = B \mathcal{S}_G A = A^* \mathcal{S}_G B = A \mathcal{S}_G B^*=  \{0\}.
\]
By definition
\[
\phi(A)\phi(B) = \sum_{i,j} \ket{f(i)}\bra{e_i} A \ket{e_i}\braket{f(i)|f(j)}\bra{e_j} B \ket{e_j}\bra{f(j)},
\]
and since $\braket{f(i)|f(j)} = 0$ whenever $i\not=j$ and $i \not\sim j$, this reduces to
\begin{align*}
\phi(A)\phi(B) &= \sum_{\substack{i,j\\ i = j \text{ or } i \sim j}} \ket{f(i)}\bra{e_i} A \ket{e_i}\braket{f(i)|f(j)}\bra{e_j} B \ket{e_j}\bra{f(j)} \\
&= \sum_{\substack{i,j\\ i = j \text{ or } i \sim j}} \braket{f(i)|f(j)} \ket{f(i)}\bra{e_i} A \ket{e_i}\bra{e_j} B \ket{e_j}\bra{f(j)}. 
\end{align*}
But when $i=j$ or $i \sim j$, we have $A\ket{e_i} \bra{e_{j}}B = 0$, and therefore
$\phi(A)\phi(B)=0$.  The same argument shows also that $\phi(B)\phi(A)=\phi(A^*)\phi(B)=\phi(A)\phi(B^*)=0$.  Therefore $\phi$ is an orthogonal representation for $\mathcal{S}_G$.
\end{proof}

\begin{proposition}
Let $G=(V,E)$ be a classical graph with $n$ vertices and let $\phi\colon M_n\to M_d$ be an orthogonal representation of $\mathcal{S}_G$.  For each $i\in [n]$, let $v_i$ be a vector in the range of $\phi(\ket{e_i}\bra{e_i})$.  Then the map $f\colon V\to \mathbb{C}^d$ defined by $f(i)=v_i$ is an orthogonal representation of $G$.
\end{proposition}

\begin{proof}
Pick any $i,j \in V$ with $i\not=j$ and $i \not\sim j$. Since $\mathcal{S}_G = \spa\{ \ket{e_k}\bra{e_\ell} \mid k \sim \ell \text{ or } k=\ell  \}$, note that $\ket{e_i}\bra{e_i} \mathcal{S}_G \ket{e_j}\bra{e_j} = \{0\}$.  As $\phi$ is an orthogonal representation, this implies $\phi( \ket{e_i}\bra{e_i} ) \perp \phi( \ket{e_j}\bra{e_j} )$ in the $C^*$-algebra $M_d$.  But then Definition \ref{def:orth} implies $f(i) \perp f(j)$ in the Hilbert space $\mathbb{C}^d$, by the way $f$ was defined.  Therefore $f$ is an orthogonal representation of $G$.
\end{proof}

\begin{remark}
The same proof as above also shows that an orthogonal representation of a quantum graph $\mathcal{S}\subseteq M_n$ induces a natural complex-valued orthogonal representation of $C_v(\mathcal{S})$ for each orthonormal basis $v$ of $\C^n$.
\end{remark}

Orthogonal representations of quantum graphs are already present in the quantum information literature.
In fact, essentially the same proof as that of Proposition \ref{prop-ortho-rep-is-order-zero} shows that if  $\phi\colon M_n\to M_d$ is a quantum channel (i.e., a completely positive and trace-preserving map), then $\phi$ is an orthogonal representation for its associated quantum confusability graph $\mathcal{S}_\phi$.
More generally, the notions of quantum (sub-)complexity of a quantum graph $\mathcal{S} \subseteq M_n$ from \cite{Levene-Paulsen-Todorov} involve considering completely positive and trace-preserving maps $\psi\colon M_n \to M_d$ whose associated quantum confusability graphs $\mathcal{S}_\psi$ are contained in $\mathcal{S}$, which by the above means that such $\psi$ are orthogonal representations for $\mathcal{S}$.

We have already observed that projections are analogues to collections of vertices when viewing quantum graphs as analogues of classical graphs.  As was the case for connectedness, this viewpoint leads to a potential candidate for a quantum definition of what it means for an orthogonal representation to be in locally general position.

Suppose $Q$ is a rank one projection in $M_n$.  If viewed as a ``quantum vertex'' of some quantum graph $\mathcal{S}\subseteq M_n$, then we also view another projection $P$ as a ``collection of vertices nonadjacent to $Q$'' if $P\mathcal{S}Q=\{0\}$.  And in this case, $\mathrm{rank}(P)$ is viewed as the ``number of vertices in $P$''.  By analogy to the classical definition, an orthogonal representation $\phi$ for $\mathcal{S}$ should preserve the rank of $P$ if it is to be viewed as being in ``locally general position''. 

\begin{definition}
Let $\mathcal{S} \subseteq M_n$ be a quantum graph, and $\phi\colon M_n \to M_d$ an orthogonal representation of $\mathcal{S}$.
We say that $\phi$ is in \emph{locally general position} if, for any fixed nonzero projection $Q\in M_n$, $\rank(\phi(P)) \ge \rank(P)$ whenever $P\in M_n$ is a projection such that $Q\mathcal{S}P = \{0\}$.
\end{definition}

Observe that in the definition above, it suffices to check the inequality for all rank one projections $Q$.
Our definition is justified by the following proposition.

\begin{proposition}
Let $G=(V,E)$ be a classical graph with $n$ vertices and let $f\colon V \to \C^d$ be an orthogonal representation of $G$ in locally general position.
Let $\phi\colon M_n \to M_d$ be the associated quantum orthogonal representation of $\mathcal{S}_G$, i.e. the mapping $\phi\colon M_n \to M_d$ given by
\[
\phi(X) = \sum_i  \ket{f(i)}\bra{e_i} X \ket{e_i}\bra{f(i)} \quad \text{ for all } X \in M_n.
\]
Then $\phi$ is in locally general position.
\end{proposition}

\begin{proof}
Fix a rank one projection $Q\in M_n$ and suppose $P\in M_n$ is a projection such that $Q\mathcal{S}_GP = \{0\}$.
Let $\ket{u} \in \C^n$ be a unit vector such that $Q = \ket{u}\bra{u}$, and observe that $\ket{u}$ is orthogonal to the range of $P$.
Let $v=\{\ket{v_j}\}_{j=1}^n$ be an orthonormal basis of $\C^n$ aligned with $P$.
By permuting the basis if necessary, we can assume that $\braket{e_i|v_i} \not=0$ for each $1\le i \le n$  \cite[Lemma 13]{Kim-Mehta}, \cite[Lemma 7.28]{paulsen2016entanglement}.  Let $i_0\in[n]$ be such that $\braket{u|e_{i_0}} \not=0$.

Let $J \subseteq [n]$ be such that $P = \sum_{j\in J} \ket{v_j}\bra{v_j}$, noting that $|J| = \rank(P)$, and for each $j \in J$, $\bra{e_j} P \ket{e_j} \not= 0$, so that in particular $\bra{e_j} P\neq 0$.
Observe that for each $j \in J$, we must have $i_0 \not\sim j$ and $i_0 \not=j$,
since otherwise we would have
\[
Q \ket{e_{i_0}}\bra{e_j}P = \braket{u|e_{i_0}} \ket{u} \bra{e_j}P \not=0,
\]
a contradiction.

Now, for any vector $\ket{x} \in \C^d$,
\begin{align*}
\phi(P)\ket{x} = 0 &\Rightarrow \bra{x}\phi(P)\ket{x} = 0\\
&\Rightarrow
\sum_i \braket{x|f(i)}\bra{e_i} P \ket{e_i}\braket{f(i)|x} = 0 \\
&\Rightarrow \sum_i |\braket{x|f(i)}|^2 \bra{e_i} P \ket{e_i} = 0 \\
&\Rightarrow \braket{x|f(i)} \cdot \bra{e_i} P \ket{e_i} = 0 \mbox{ for every }i\in [n]\\
&\Rightarrow \braket{x|f(j)}=0 \mbox{ for every }j\in J.
\end{align*}
That is,
\[
\ker\big( \phi(P) \big) \subseteq \big( \spa\{ f(j) \mid j \in J \} \big)^\perp.
\]
Because $f$ is in locally general position, and the indices in $J$ correspond to vertices in $G$ which are not adjacent to $i_0$, the dimension of $\spa\{ f(j) \mid j \in J \}$ is exactly $|J|$ and $|J|\leq d$. Therefore
\[
\dim\big( \ker\big( \phi(P) \big) \big) \le d -\rank(P),
\]
from where
\[
\rank\big( \phi(P) \big) \ge \rank(P).
\]
That is, $\phi$ is in locally general position.
\end{proof}

Finally, we arrive at the main result of this section, which shows that some connectivity of a quantum graph can be inferred from the existence of an orthogonal representation in locally general position, in analogy to the classical result.

\begin{proposition}
Let $\mathcal{S} \subseteq M_n$ be a quantum graph, and suppose there exists an orthogonal representation $\phi : M_n \to M_d$ of $\mathcal{S}$ in locally general position. Then $\mathcal{S}$ is $(n-d)$-connected.
\end{proposition}

\begin{proof}
If every separator of $\mathcal{S}$ has rank greater than or equal to $n-1$, then $\mathcal{S}$ is $(n-1)$-connected and so also $(n-d)$-connected.  So suppose there is a separator $P$ for $\mathcal{S}$ such that $\mathrm{rank}(P)<n-1$, so that
there exist nontrivial disjoint projections $Q_1, Q_2$, also disjoint from $P$, such that $I_n = P + Q_1 + Q_2$ and $Q_1\mathcal{S}Q_2 = \{0\}$.
Since $\phi$ is an orthogonal representation of $\mathcal{S}$,
 $\phi(Q_1)\perp \phi(Q_2)$ in the $C^*$-algebra $M_d$, and thus
\[
d \ge \rank(\phi(Q_1)) + \rank(\phi(Q_2)).
\]
And since $\phi$ is in locally general position, it follows
that $\rank(\phi(Q_j)) \ge \rank(Q_j)$, so
\[
d \ge \rank(Q_1) + \rank(Q_2) = n - \rank(P),
\]
and so $\rank(P) \ge n-d$.  Therefore $\mathcal{S}$ is $(n-d)$-connected.
\end{proof}

It would be really interesting to know whether the opposite implication holds, that is, whether a certain amount of connectivity implies the existence of an orthogonal representation in locally general position of the appropriate size.
We point out that this does hold in the case of maximal connectivity:
 If $\mathcal{S} \subseteq M_n$ is $(n-1)$-connected, it follows from Proposition \ref{prop-maximal-connectivity} that the trace $\tr\colon M_n \to M_1=\C$ is an orthogonal representation in locally general position for $\mathcal{S}$ (because the required conditions are vacuously satisfied).



\bibliography{references}
\bibliographystyle{amsalpha}

\end{document}